\documentclass[12pt,draft]{article}

\usepackage{amsmath, amsfonts, amssymb}
\usepackage{mathrsfs}
\usepackage{theorem}
\usepackage{bm}
\RequirePackage[colorlinks,citecolor=blue,urlcolor=blue,bookmarksopen]{hyperref}

\newtheorem{mythm}{Theorem}[section]
\newtheorem{myprop}[mythm]{Proposition}

\newtheorem{mydefn}[mythm]{Definition}

{  \newtheorem{myrem}[mythm]{Remark}}
{ }

\newcommand{\ra}{\rightarrow}
\newcommand{\dis}{\displaystyle}

\def\R{\mathbb R}

\def\N{\mathbb N}
\def\C{\mathscr C}
\def\B{\mathscr B}

\def\F{\mathscr F}
\def\d{\text{\rm{d}}}
\def\E{\mathbb E}
\def\p{\mathbb P}
\def\e{\text{\rm{e}}}

\def\La{\Lambda}
\def\veps{\varepsilon}

\def\ll{\mathscr{L}}

\def\S{\mathcal S}
\def\C{\mathscr C}

\def\pb{\mathscr{P}}

\def\wt{\widetilde}

\newcommand{\fin}{\hfill $\square$\par}
\newenvironment{proof}{{\noindent\it Proof.}\ }{\hfill $\square$\par}

\numberwithin{equation}{section}

\allowdisplaybreaks

\begin{document}

\title{The existence of optimal control for continuous-time Markov decision processes in random environments\footnote{Supported in
 part by NNSFs of China (No. 11771327,  11831014)}}

\author{Jinghai Shao\thanks{Center for Applied Mathematics, Tianjin University, Tianjin 300072, China. Email: shaojh@tju.edu.cn.} and Kun Zhao \thanks{Center for Applied Mathematics, Tianjin University, Tianjin 300072, China.}
}

\maketitle

\begin{abstract}
  In this work, we investigate the optimal control problem for continuous-time Markov decision processes with the random impact of the environment. We provide conditions to show the existence of optimal controls under finite-horizon criteria. Under appropriate conditions, the value function is continuous and satisfies the dynamic programming principle. These results are established by introducing some restriction on the regularity of the optimal controls and by developing a new compactification method for continuous-time Markov decision processes, which is originally used to solve the optimal control problem for jump-diffusion processes.
\end{abstract}



\section{Introduction}
Continuous-time Markov decision processes (CTMDPs) have been extensively studied and widely applied in various application fields such as telecommunication, queueing systems, population processes, epidemiology, and so on.  See, for instance, the monographs \cite{GH, PH}, the works
\cite{Guo07, GH03, Guo15, GR, GVZ, KS, Mil, PZ} and references therein.  As an illustrative example, we consider the controlled queueing systems.  In a single-server queueing system, jobs or customers arrive, enter the queue, wait for service, receive service, and then leave the system. A decision-maker can control the system by deciding which jobs to be admitted to the queue, by  increasing or decreasing the arrival rates or service rates in order to maximize the reward or minimize the cost of this system. There are many researches on CTMDPs under various optimality criteria. For example, the expected discounted, average and the finite-horizon optimality criteria have been well studied in \cite{GH, PH} and \cite{Guo15, Mil, Yu} amongst others.

However, in realistic applications, the cost of raw materials or the price of products depends on not only the number of jobs or customers but also the prices of raw materials or  products.
In this work, we shall extend the classical CTMDPs to make these models more realistic by including the random effect of the market. A diffusion process on $\R^d$ is included to model the price process whose coefficients may be dependent on the continuous-time Markov chain. A decision-maker still controls the system by deciding the transition rate of the Markov chain, but the optimality criterion depends on both the diffusion process and the Markov chain. The coexistence of Markov chains and   diffusion processes makes the optimality problem more difficult.  The well developed methods in the study of CTMDPs such as in \cite{GH} and \cite{Guo07, Guo15} do not work anymore.
For instance,  to deal with the infinite horizon expected discounted reward, it is quite crucial to establish the optimality equation based on the recursion approximation of the Laplace transform for the continuous-time Markov chain; see \cite[Theorem 4.6]{GH} and
\cite[p.121-122]{And}. Nevertheless, the appearance of the second order differential operators associated with the diffusion process makes it harder to first establish the optimality equation and then to show the existence of the optimal control.

In this work, we develop a compactification method to provide some sufficient conditions on the existence of optimal controls. This kind of compactification method is usually used to study the optimal control problem for jump-diffusion processes, and
has been well studied by many works including \cite{Chow, DM, HL90, HS95, HS95b, KS, Ku}. See \cite{HS95} for a complete list of references on the subject. In order to deal with CTMDPs in a random environment, we introduce $\psi$-relaxed controls as the class of admissible controls. The function $\psi$ is used to characterize the regularity of the optimal controls. The class of $\psi$-relaxed controls contains all randomized stationary policies in some sense (see Section 2 for details). The randomized stationary policies have been extensively investigated in the study of CTMDPs; see for example the monograph \cite{GH}.  The basic idea of our method is similar to that of Haussmann and Suo \cite{HS95}, but there are some essential differences on the measurability of the control policies. In \cite{HS95}, the controllers are assumed to have no information on the state of the studied system, so the admissible control policies are all adapted to some given $\sigma$-fields. However, to deal with CTMDPs, the control policy must be adapted to the $\sigma$-fields generated by the Markov chain in order to keep the Markovian property of the studied system.
Therefore, the key difficulty of this work is to show that the jumping process remains to be a Markov chain under all admissible controls in current situation. Besides, concrete techniques raised in this work are also different to those in \cite{HS95}. This can be reflected by the fact that this work can treat the terminal cost, however,  \cite{HS95} cannot (cf. \cite[Remark 2.2]{HS95}).

To be more precise, consider a Markov chain $(\La_t)$ on a denumerable state space $\S$ associated with the transition rate $q$-pair $(q(\theta,A;u),q(\theta;u))$, where $\theta\in\S$, $A\in \B(\S)$, $u\in U$, and the action set $U$ is a compact subset of $\R^k$. Let us consider further a diffusion process $(X_t)$ satisfying the following stochastic differential equation (SDE):
\begin{equation}\label{1.1}
\d X_t=b(X_t,\La_t)\d t+\sigma(X_t,\La_t)\d B_t,
\end{equation} where $b:\R^d\times \S\ra \R^d$, $\sigma:\R^d\times \S\ra \R^{d\times d}$, and $(B_t)$ is a standard $d$-dimensional Brownian motion. The process $(X_t)$ is used to model the price of raw materials or products, which is related not only to the randomness of the market characterized by the Brownian motion, but also to the number of jobs or the customers characterized by the Markov chain $(\La_t)$. Relaxed controls, known also as randomized policies, are considered in this paper. The following finite-horizon criterion is used:
\[\E\Big[\int_0^T f(t,X_t,\La_t,\mu_t)\d t+g(X_T,\La_T)\Big],\]
where $f:[0,T]\times \R^d\times\S\times U\to \R$ and $g:\R^d\times\S\to \R$ stand for the cost functions. Here and in the remainder of this paper, a measurable function $h:U\ra \R$ is extended into a function on $\pb(U)$, the collection of all probability measures on $U$, through:
\[h(\mu):=\int_{U}h(u)\mu(\d u),\quad \mu\in \pb(U),\]
whenever the integral is well defined.

Our contribution of this paper consists of two aspects: one is to include the random impact of the environment into the cost/reward function to provide more realistic models than classical CTMDPs in applications; another is to propose a new method to study the existence of optimal controls for CTMDPs, which generalizes the method of
\cite{HL90, HS95, HS95b, Ku} in the setting of Markov chains. This method can also be generalized to deal with the history-dependent control problem investigated in
\cite{Guo15}, where the existence of optimal history-dependent control was left open.  Moreover, the concept of $\psi$-relaxed control proposed in this work is of interest by itself, which is closely related to the well studied randomized policy (cf. e.g. \cite{GH}). See the subsection 2.1 below for the details.

This work is organized as follows: To focus on the development of compactification method in \cite{HS95, Ku} from the setting of diffusion processes to that of CTMDPs, we consider in Section 2 only the optimal control problem for classical CTMDPs without any random impact of the environment.
In Section 3 we treat CTMDPs in a random environment, and show the existence of the optimal control under appropriate conditions.

\section{Optimal Markov control for CTMDPs}
In this part we aim to develop the compactification method in \cite{HS95, Ku} from the setting of jump-diffusion processes to the setting of CTMDPs. To focus on this development and simplify the representation, we do not consider the impact of random environment in this section.  We introduce the concept of $\psi$-relaxed control to ensure the Markovian property of the studied system, and discuss its  connection with the classical randomized control policies studied, for instance, in \cite{GH,Guo15,GVZ}. In short, the class of $\psi$-relaxed controls is a subset of general randomized control policies in some sense, but contains all the randomized stationary policies and deterministic stationary policies. Randomized or deterministic stationary policies are two important kinds of policies having been extensively studied in  \cite{Guo07, GH03,GH, Guo15, GVZ, PZ, PH} amongst others. In these works, many obtained optimal control policies are all stationary.

\subsection{Formulation and Assumptions}

Let $(\Omega,\mathscr F,\p)$ be a probability space with the filtration $\{\mathscr F_t\}_{t\geq0}$. $\{\mathscr F_t\}$ satisfies the usual condition, that is, $\F_t$  is right-continuous and $\mathscr F_0$ contains all the $\p$-negligible events in $\F$.
Let $\S$ be a countable state space.
Let $U\subset \R^k$ be a compact set, and $\pb(U)$ the collection of all probability measures over $U$. On $\pb(U)$, define the $L_1$-Wasserstein distance between two probability measures $\mu$ and $\nu$ by:
\begin{equation}\label{wass}
W_1(\mu,\nu)=\inf\Big\{\int_{U\!\times \!U}\!|x-y|\pi(\d x,\d y); \,\pi\in \C(\mu,\nu)\Big\},
\end{equation}
where $\C(\mu,\nu)$ stands for the collection of all probability measures on $U\times U$ with marginal $\mu$ and $\nu$ respectively. Since $U$ is compact, and hence is bounded, the weak topology of $\pb(U)$ is equivalent to the topology induced by the $L_1$-Wasserstein distance. Also, this implies that $(\pb(U),W_1)$ is a compact Polish space (cf. \cite[Chapter 7]{AGS}).
We focus on  the finite-horizon optimal control problem in this work, so let us fix a time $T>0$ throughout this work.

Let $\S$ be a denumerable state space endowed with discrete topology.
Given $u\in U$, we call  $(q(\theta;u),q(\theta, A;u))$ $(\theta\in \S,\,A\in \B(\S))$ a \emph{$q$-pair}, if for each $A\in\B(S)$, $\theta\mapsto q(\theta; u)$ and $\theta\mapsto q(\theta, A;u)$ are measurable; and for each $\theta\in \S$, $A\mapsto q(\theta,A;u)$ is a measure on $\S$, $q(\theta,\{\theta\};u)=0$, $q(\theta,\S;u)\leq q(\theta;u)$. Moreover, it is called \emph{conservative} if $q(\theta;u)=q(\theta,\S;u)$ for all $\theta\in\S$. A function $h:\S\ra [0,\infty)$ is called a compact function if for every $\alpha>0$, the set $\{\theta\in\S; h(\theta)\leq \alpha\}$ is compact.

In the following we collect the hypotheses used in this section:
\begin{itemize}
  \item[(H1)] $U\subset \R^k$ is  a compact set for some $k\in \N$.
  \item[(H2)] For each $u\in U$, $(q(\theta;u),q(\theta,A;u))$ is a conservative $q$-pair on $\S$. Moreover, $M:=\sup_{u\in U}\sup_{\theta\in \S} q(\theta, \S;u)<\infty$.
  \item[(H3)] For every $\theta\in\S$ and $A\in \B(\S)$, the function $u\mapsto q(\theta,A;u)$ is continuous on $U$. For every $A\in \B(\S)$, $u\in U$, the function $\theta\mapsto q(\theta, A;u)$ is continuous.
  \item[(H4)] There exist a compact function $\Phi:\S\to [1,\infty)$, a compact set $B_0\in \mathscr B(\S)$, constants $\lambda>0$ and $\kappa_0<\infty$ such that
      \begin{equation*}
      \begin{split}
      Q_u\Phi(\theta)&:=\int_{\S}q(\theta,\d \gamma;u)\Phi(\gamma)-q(\theta;u)\Phi(\theta)\\
      &\leq \lambda \Phi(\theta)+\kappa_0\mathbf 1_{B_0}(\theta), \qquad \theta\in \S,\ u\in U.
      \end{split}
      \end{equation*}
\end{itemize}

\begin{myrem} The boundedness of $q(\theta,\S;u)$  in $\mathrm{(H2)}$ ensures that the jumping process $(\La_t)$ owns almost surely finite number of jumping in every finite time interval. As an initiative investigation to include the random effect of the environment to the theory of CTMDPs, we impose simply the bounded condition $\mathrm{(H2)}$ of the transition rates. In the study of CTMDPs, there are some works to deal with unbounded transition rates.  For example, in \cite{Guo15},  the authors used a technique of approximations from bounded transition rates to unbounded ones to establish the existence of optimal Markovian controls.  $\mathrm{(H4)}$ is called a drift condition, which is used to guarantee the non-explosion of the process $(\La_t)$ and to prove the tightness of the distributions of the Markov chains.
\end{myrem}

Let $\psi:[0,T]\ra [0,\infty)$ be an increasing function such that
\begin{equation}\label{psi}
\lim_{r\ra 0}\psi(r)=0.
\end{equation}
Consider the space $\mathcal{D}([0,T];\pb(U))$ of measurable maps from $[0,T]$ to the Polish space $(\pb(U),W_1)$ that are right-continuous with left-limits. Endow $\mathcal{D}([0,T];\pb(U))$ with the Skorokhod topology, which makes $\mathcal{D}([0,T];\pb(U))$ a Polish space; see \cite{Bill}.
For $\mu:[0,T]\ra \pb(U)$ in $\mathcal{D}([0,T];\pb(U))$,
put
\[w_\mu([a,b))=\sup\{W_1(\mu_t,\mu_s);\ s,t\in [a,b)\}, \quad a,b\in [0,T], a<b.\]
To describe compact sets in $\mathcal{D}([0,T];\pb(U))$, let us introduce the function
\begin{equation}\label{w2}
w_\mu''(\delta)=\sup \min\big\{ W_1(\mu_t,\mu_{t_1}),W_1(\mu_t,\mu_{t_2})\big\},
\end{equation}
where the supremum is taken over $t_1$, $t$, and $t_2$ satisfying
\[t_1\leq t\leq t_2,\qquad t_2-t_1\leq \delta.\]

\begin{mydefn}\label{def-1}
A $\psi$-relaxed control is a term $\alpha=(\Omega,\F,\F_t,\p,\La_t,\mu_t,s,\theta)$ satisfying:
\begin{itemize}
  \item[$\mathrm{(1)}$] $(s,\theta)\in [0,T]\times\S$;
  \item[$\mathrm{(2)}$] $(\Omega,\F,\p)$ is a probability space with the filtration $\{\F_t\}_{t\in [0,T]}$;
  \item[$\mathrm{(3)}$] $\mu_t\in \pb(U)$ is adapted to the $\sigma$-field generated by $\La_t$,
       $t\mapsto \mu_t$ is in $\mathcal{D}([0,T];\pb(U))$ almost surely, and for every $\theta'\in \S$ the curve $t\mapsto\nu_t(\,\cdot,\theta'):=\mu_t(\, \cdot\, |\La_t=\theta')$ satisfies
      \[w_{\nu}([t_1,t_2))\leq \psi( t_2-t_1),\quad 0\leq t_1<t_2\leq T;\]
  \item[$\mathrm{(4)}$] $(\La_t)_{t\in [s,T]}$ is an $\F_t$-adapted, jumping process with $\La_s=\theta$ and satisfies
   \begin{equation}\label{MD-1}
   \p(\La_{t+\delta}\in A|\La_t=\theta,\mu_t=\mu)-\mathbf 1_{A}(\theta)=\big(q(\theta,A;\mu)-q(\theta;\mu)\mathbf 1_A(\theta)\big)\delta+o(\delta)
   \end{equation} provided $\delta>0$.
\end{itemize}
\end{mydefn}
The collection of all $\psi$-relaxed controls with initial value $(s,\theta)$ is denoted by $\wt \Pi_{s,\theta}$. The function $\psi$ is used to characterize the regularity of the optimal controls.

The set $\wt \Pi_{s,\theta}$ consists of many interesting and well studied controls. We proceed to show that all the randomized stationary policies and deterministic stationary policies studied, for example,  in  \cite{Guo07, GH, Guo15, GVZ} are all associated with $\psi$-relaxed controls in a natural way.

Recall the definition of randomized Markov policies from \cite{GH}.  A randomized Markov policy is a real-valued function $\pi_t(C|\theta')$ that satisfies the following conditions:
\begin{itemize}
  \item[$\mathrm{(i)}$] For all $\theta'\in\S$ and $C\in \B(U)$, $t\mapsto \pi_t(C|\theta')$ is measurable on $[0,\infty)$.
  \item[$\mathrm{(ii)}$] For all $\theta'\in\S$ and $t\geq 0$, $C\mapsto \pi_t(C|\theta')$ is a probability measure on $\B(U)$, where $\pi_t(C|\theta')$ denotes the probability that an action in $C$ is taken when the system's state is $\theta'$ at time $t$.
\end{itemize}
A randomized Markov policy $\pi_t(\d u|\theta')$ is said to be stationary if $\pi_t(\d u|\theta')$ is independent of time $t$.

For any $\psi$-relaxed control $\alpha=(\Omega,\F,\F_t,\p,\La_t,\mu_t,s,\theta)$, we shall show that $\mu_t$ indeed acts as a randomized Markov policy $\pi_t(C|\theta)$. Firstly,
since $\mu_t$ is adapted to the $\sigma$-field generated by $\La_t$ according to Definition \ref{def-1}, this yields that there exists a measurable map $F_t:\S\ra \pb(U)$ such that $\mu_t=F_t(\La_t)$. (This is a result derived from the functional monotone class theorem in measure theory.) Thus, if $\La_t=\theta'$ is given, then $\mu_t=F_t(\theta')$ is a fixed probability measure in $\pb(U)$. We may rewrite $\mu_t$ as
\begin{equation}\label{mu}
\mu_t(\d u)=\sum_{\theta'\in\S} F_t(\theta')(\d u)\mathbf{1}_{\{\La_t=\theta'\}}.
\end{equation}
Condition (3) of Definition \ref{def-1} ensures that $F_t(\theta')$ is right-continuous with left-limits.  So $\pi_t(\d u|\theta'):=F_t(\theta')(\d u)$ satisfies the conditions (i) and (ii) of a randomized Markov policy. Consequently, the class of $\psi$-relaxed controls is a subclass of randomized Markov policies in some sense.

Moreover, for a randomized stationary policy $\pi(\d u|\theta')$, let
\begin{equation}\label{trans-1}
\tilde \mu_t=\sum_{\theta'\in \S}\pi(\d u|\theta')\mathbf{1}_{\La_t=\theta'}, \quad t\in [0,T].
\end{equation} According to the path property of continuous-time Markov chains, it is clear that $(\tilde \mu_t)$ defined by \eqref{trans-1} satisfies the condition (3) of Definition \ref{def-1} with $\nu_t(\d u,\theta')=\pi(\d u|\theta')$ for all $t\geq 0$ and $\theta'\in\S$. Hence, $W_\nu([t_1,t_2))=0$ for every $0\leq t_1<t_2$. Corresponding to the randomized stationary Markov policy $\pi(\d u|\theta')$, there exists a CTMDPs $(\La_t)$ in some probability space $(\Omega,\F,\F_t,\p)$ with initial value $\La_s=\theta$; see \cite[Chapter 2]{GH}. It follows immediately that $(\Omega,\F,\F_t,\p,\La_t,\tilde \mu_t,s,\theta)$ is a $\psi$-relaxed control for any $\psi$ satisfying \eqref{psi}.  By viewing a deterministic stationary policy $\xi:\S\ra U$ as a randomized policy $\pi:\S\ra \pb(U)$ through the transform
\[\pi(\d u|\theta')=\mathbf{1}_{\xi(\theta')}(\d u),\]
we know that every deterministic stationary policy is corresponding to a $\psi$-relaxed control.

Conditions (3) and (4) of Definition \ref{def-1} also tell us that the transition rate does not depend on the past of the process $(\La_t)$, so the process $(\La_t)$ is indeed a Markov process. Put
\begin{equation}\label{m-q}
q(t,\theta',A)=\E\Big[\!\int_U\!\! q(\theta',A;u)\mu_t(\d u)\big|\La_t=\theta'\Big],\  q(t,\theta')=\E\Big[\!\int_U \!\!q(\theta';u)\mu_t(\d u)\big|\La_t=\theta'\Big]
\end{equation} for $A\in \B(\S)$,
then the transition probability of the process $(\La_t)$ satisfies
\begin{equation}\label{m-0}
\p(\La_{t+\delta}\in A|\La_t=\theta')-\mathbf 1_A(\theta')=\big(q(t,\theta',A)-q(t,\theta')\mathbf 1_A(\theta')\big)\delta+o(\delta).
\end{equation}

Given two measurable functions $f:[0,T] \times \S\times U\to \R$ and $g:\S\ra \R$, the expected cost under the policy $\tilde \mu\in \wt\Pi$ is defined by
\begin{equation}\label{m-1}
J(s,\theta, \alpha)=\E\Big[\int_s^T f(t,\La_t,\mu_t)\d t+g(\La_T)\Big], \quad s\in [0,T), \,\theta\in \S.
\end{equation}
Define the value function by
\begin{equation}\label{m-2}
V(s,\theta)=\inf_{ \alpha\in  \wt \Pi_{s,\theta}} J(s,\theta,\alpha),\quad s\in [0,T), \,\theta\in \S.
\end{equation}
For $s\in [0,T],\, \theta\in\S$, a $\psi$-relaxed control   $\alpha^\ast\in \wt \Pi_{s,\theta}$ is called   \emph{optimal} if
\begin{equation}\label{m-5}
V(s,\theta)= J(s,\theta, \alpha^\ast).
\end{equation}

\subsection{Existence of optimal control}
After the preparation of the previous subsection, we can state our result on the existence of optimal $\psi$-relaxed controls. We shall follow Haussmann and Suo's approach, and one can refer to \cite{Guo15} for alternative approach in the setting of CTMDPs without the random impact of the environment.

\begin{mythm}\label{tm-1}
Given $T>0$, assume (H1)-(H4) hold.
Suppose $f$ and $g$ are lower semi-continuous and bounded from below. Then for every $s\in [0,T)$ and $\theta\in \S$ there exists an optimal $\psi$-relaxed control $\alpha^\ast\in\wt \Pi_{s,\theta}$.
\end{mythm}

Before proving this theorem, for a  relaxed control $(\Omega,\F,$ $\F_t,\p,\La_t,\mu_t,s,\theta)$ we provide a representation of the transition probability of the Markov chain $(\La_t)$. Define
\begin{equation}\label{m-7}
P_{s,t}^{ \mu}\mathbf 1_{A}(\theta)=P^{\mu}(s,\theta,t,A)=\p(\La_t\in A|\La_s=\theta),\quad \theta\in \S,\, A\in \B(\S),
\end{equation}
and
\begin{equation}\label{m-8}
 Q^{ \mu}(t)h(\theta)=\int_{\S}\!  q(t,\theta,\d \gamma)h(\gamma)-  q(t,\theta)h(\theta),\quad h\in \B_b(\S),
\end{equation} where $q(t,\theta,\cdot)$ and $q(t,\theta)$ are given by \eqref{m-q}, $\B(\S)$ denotes the set of measurable functions on $\S$, and $\B_b(\S)$ is the set of bounded measurable functions on $\S$.

%

\begin{myprop}\label{tm-0}For a  relaxed control $(\Omega,\F,\F_t,\p,\La_t,\mu_t,s,\theta)$, it holds, for $h\in \B_b(\S)$,
\begin{equation}\label{m-9}
\begin{split}
P_{s,t}^{  \mu} h(\theta)&=h(\theta)+\int_s^t   Q^{  \mu}(t_1)h(\theta)\d t_1+\int_s^{t}\!\int_{s}^{t_2}\!  Q^{  \mu}(t_2)  Q^{ \mu}(t_1)h(\theta)\d t_1\d t_2\\
&\quad+\sum_{n=3}^\infty\!\int_s^{t}\!\int_s^{t_n}\!\cdots\!\int_s^{t_2}  Q^{  \mu}(t_n)  Q^{ \mu}(t_{n-1})\cdots  Q^{ \mu}(t_1)h(\theta)\d t_1\ldots\d t_{n-1}\d t_n.
\end{split}
\end{equation}
\end{myprop}

\begin{proof}
Due to (4) of Definition \ref{def-1} and \eqref{m-q}, \eqref{m-0}, we know that $(\La_t)$ is a time-inhomogeneous Markov process. Therefore,
\begin{align*}
  P_{s,t+\delta}^{\mu}h(\theta)=P_{s,t}^{\mu}P_{t,t+\delta}^\mu h(\theta),\quad  h\in \B_b(\S).
\end{align*}
Invoking \eqref{m-0}, this yields the equation
\begin{equation}\label{m-10}
\frac{\d}{\d t}P_{s,t}^\mu h(\theta')=P_{s,t}^\mu Q^{\mu}(t)h(\theta'),\quad P_{s,s}^\mu h(\theta')=h(\theta'),\quad \theta'\in \S, h\in \B_b(\S).
\end{equation}
See, e.g. \cite{Feller} for more details on this deduction.
Thus, according to \cite[Chapter III]{DK}, formulae (1.12) and (1.15) therein, the unique solution of \eqref{m-10} has an explicit representation \eqref{m-9} in terms of the Cauchy operator.

Let us show the series in \eqref{m-9} is well defined. Endowed with the essential supremum norm $\|\cdot\|_\infty$, $\B_b(\S)$ becomes a Banach space. Viewed as a linear operator over $\B_b(\S)$, define the operator norm of $Q^\mu(t)$ by:
\[\|Q^\mu(t)\|=\sup_{\|h\|_\infty\leq 1} \|Q^\mu(t)h\|_\infty,\]
which  obviously satisfies
\[\|Q^\mu(t)\|\leq \sup_{\theta\in \S}\sup_{u\in U} 2q(\theta;u)\leq 2M<\infty,\quad \forall t\in [0,T].\] Hence,
\begin{equation} \label{m-11}
  \begin{split}
    &\Big|\int_s^t\!\int_s^{t_n}\ldots\int_{s}^{t_2}\!Q^{  \mu}(t_n)Q^{  \mu}(t_{n-1})\ldots Q^{ \mu}(t_1)h(\theta)\d t_1\ldots\d t_{n-1}\d t_n\Big|\\
    &\leq \|h\|_\infty\!\int_s^t\!\int_s^{t_n}\!\ldots\int_s^{t_2}\|Q^{  \mu}(t_n)\|\|Q^{  \mu}(t_{n-1})\|\ldots\|Q^{  \mu}(t_1)\|\d t_1\ldots\d t_{n-1}\d t_{n}\\
    &=\frac{\|h\|_\infty}{n!}\int_s^t\!\int_s^{t}\!\ldots\int_s^{t}\|Q^{  \mu}(t_n)\|\|Q^{  \mu}(t_{n-1})\|\ldots\|Q^{  \mu}(t_1)\|\d t_1\ldots\d t_{n-1}\d t_{n}\\
    &=\frac{\|h\|_\infty}{n!}\Big(\int_s^t\|Q^{ \mu}(r)\|\d r\Big)^n\\
    &\leq \frac{(2M(t-s))^n}{n!} \|h\|_\infty,
    \end{split}
  \end{equation} since the integral is invariant under any perturbation of the variables $t_1,\ldots,t_n$.
  Therefore, the series in \eqref{m-9} is convergent, and further the operator $P_{s,t}^\mu$ is well defined.
\end{proof}

Just as done in \cite{HS95}, the relaxed controls  can be transformed into controls in the canonical path space to simplify the arguments.
Let
\begin{equation}\label{U}
\mathcal U=\{\nu:[0,T]\ra \pb(U);\ \nu\in \mathcal{D}([0,T];\pb(U)),\ w_\nu''(\delta)\leq \psi(\delta), \ \delta\in (0,T]\},
\end{equation}
which is viewed as a subspace of $\mathcal{D}([0,T];\pb(U))$.
Denote
\[\mathcal D([0,T];\S)=\big\{y:[0,T]\ra \S \,\text{is right-continuous with left-limits}\big\},\]
which is a Polish space endowed with Skorokhod topology.
Consider the canonical space
\[\mathcal Y=\mathcal{D}([0,T];\S)\times \mathcal U.\]
Let $\wt{ \mathcal{D}}$, $\wt{ \mathcal{U}}$ be their Borel $\sigma$-fields, and $\wt{\mathcal D}_t$,
$\wt{ \mathcal U}_t$ the $\sigma$-fields up to time $t$.
Put
\[\wt{ \mathcal Y}=\wt{ \mathcal D}\times \wt{ \mathcal U},\quad \wt{ \mathcal Y}_t=\wt{ \mathcal D}_t\times\wt{  \mathcal U}_t.\]
Then, every $\psi$-relaxed control $(\Omega,\F,\F_t,\p,\La_t,\mu_t,s,\theta)$ can be transformed into a new $\psi$-relaxed control $(\mathcal Y,\wt{\mathcal Y},\wt{\mathcal Y}_t,R,\La_t,\mu_t,s,\theta)$ via the map
$\Psi:\Omega\ra \mathcal Y$ defined by
\[\Psi(\omega)=(\La_t(\omega),\mu_t(\omega))_{t\in [0,T]},\quad \La_r:=\theta,\ \mu_r:=\mu_s,\ \forall\,r\in[0,s],\]
where
$\dis R=\p\circ\Psi^{-1}$ is a probability measure on $\mathcal Y$. Similar to the discussion in \cite{HS95}, it is clear that the $\psi$-relaxed control  $\alpha=(\mathcal Y,\wt{\mathcal Y},\wt{\mathcal Y}_t,R,\La_t,\mu_t,s,\theta)$ is completely determined by the probability measure $R$, so in the canonical space we use $R$ itself to denote this $\psi$-relaxed control $\alpha$.

\noindent\textbf{Proof of Theorem \ref{tm-1}}\quad
If $ V(s,\theta)=\infty$, then every $\psi$-relaxed control $\alpha$ will be optimal. So, we only need to consider the case $ V(s,\theta)<\infty$.  We only consider the case $s=0$ to simplify the notation. The proof is separated into three steps.

\textbf{Step 1}. According to the definition of $ V(0,\theta)$ and previously introduced representation of $\psi$-relaxed controls on the canonical space, there exists a sequence of probability measures $R_n$, $n\geq 1$, on $\mathcal Y$ such that
\begin{equation}\label{vn}
\lim_{n\ra \infty}  J(0,\theta,R_n)= V(0,\theta)<\infty.
\end{equation}
In this step, we aim to prove that  $(R_n)_{n\geq 1}$ is tight. To this end, let $\ll_\La^n$ and $\ll_\mu^n$, $n\geq 1$, the marginal distribution of $(\La_t)_{t\in [0,T]}$ and $(\mu_t)_{t\in [0,T]}$ respectively under $R_n$.

Since $U$ is a compact set, $(\pb(U), W_1)$ is a compact Polish space. Then,  according to \cite[Theorem 14.3]{Bill} or \cite[Theorem 6.3]{EK},
$\mathcal U$ is a compact subset in $\mathcal{D}([0,T];\pb(U))$. Moreover, by the definition of $\psi$-relaxed control, $\mu$ admits a representation \eqref{mu}, and $F_t(\theta')$ is in $\mathcal U$ for every $\theta'\in\S$. The compactness of $\pb(U)$ implies the boundedness of $\pb(U)$, i.e. there exists a constant $K>0$ such that $W_1(\nu_1,\nu_2)\leq K$ for any $\nu_1,\,\nu_2\in\pb(U)$. This yields immediately that for some fixed $\nu\in \pb(U)$,
\[R_n\big(\omega: \sup_{0\leq t\leq T}W_1(\mu_t, \nu)>K\big)= 0,\quad n\geq 1.\]

We go to estimate $R_n(\omega:w_{\mu(\omega)}''(\delta)\geq \veps)$, $n\geq 1$. For any $\veps\in(0,K)$,  there exists a $\delta>0$ such that $\psi(\delta)<\veps$. According to Definition \ref{def-1}, for every $\theta'\in\S$, denoting by $\nu_t(\cdot,\theta'):=\mu_t(\cdot|\La_t=\theta')$, it holds
\[w_\nu([t_1,t_2))\leq \psi(t_2-t_1)\leq \psi(\delta)<\veps, \ \ 0\leq t_1<t_2\leq T,\ t_2-t_1\leq \delta.\]
Also, we can rewrite $\mu_t(\cdot)=\nu_t(\cdot,\La_t)$. By the triangle inequality,
\begin{align*}
W_1(\mu_t,\mu_{t_1})&\leq W_1(\nu_t(\cdot,\La_t),\nu_{t_1}(\cdot, \La_t))+W_1(\nu_{t_1}(\cdot,\La_t),\nu_{t_1}(\cdot,\La_{t_1}))
\\
&\leq W_1(\nu_t(\cdot,\La_t),\nu_{t_1}(\cdot, \La_t)) +K\mathbf{1}_{\La_t\neq \La_{t_1}}.
\end{align*}
Hence, for any $t_1,\,t, t_2\in [0,T]$ with $t_1\leq t\leq t_2$ and $t_2-t_1\leq \delta$, if there exist no more than two jumps for the Markov chain $(\La_t)$ during the time period $[t_1, t_2]$, it must hold
\begin{align*}
&\min \{W_1(\mu_{t_1},\mu_t), W_1(\mu_{t_2},\mu_t)\}\\
&\leq \min\{W_1(\nu_{t}(\cdot,\La_t),\nu_{t_1}(\cdot,\La_t))+K\mathbf 1_{\La_t\neq \La_{t_1}}, W_1(\nu_{t}(\cdot,\La_t),\nu_{t_2}(\cdot,\La_t))+K\mathbf 1_{\La_t\neq \La_{t_2}}\}  <\veps.
\end{align*}
Thus,
\begin{equation}\label{et-1}
\begin{split}
&R_n\big(\omega:\ \min\{W_1(\mu_{t_1},\mu_t),W_1(\mu_{t_2},\mu_t)\}\geq \veps\big)\\
&\leq R_n(\omega: \ \text{the process $(\La_r)$ owns at least two jumps during $[t_1,t_2]$})\\
&\leq o(\delta).
\end{split}
\end{equation}
Moreover, the arbitrariness of $t_1,\,t,\,t_2$ implies that for each positive $\veps$ and $\eta$, there exists $\delta\in (0,T)$ such that
\begin{equation}\label{vn-1}
\begin{split}
R_n(\omega: w_\mu''(\delta)\geq \veps)\leq o(\delta)\leq \eta.
\end{split}
\end{equation} For the Markov chain $(\La_t)$ with the bounded transition rate matrices, it is clear that for $\delta>0$ sufficiently small,
\begin{equation}\label{et-2}
R_n(\omega: w_\mu([0,\delta))\geq \veps)\leq \eta,\qquad R_n(\omega: w_\mu([T-\delta,T))\geq \veps)\leq \eta,\quad n\geq 1.
\end{equation}
Applying \cite[Theorem 15.3]{Bill}, we show that $(\ll_\mu^n)_{n\geq 1}$ is tight.

Next, we go to prove the set of probability measures $(\ll_\La^n)_{n\geq 1}$ on $\mathcal D([0,T];\S)$ is tight. We shall apply Kurtz's tightness criterion
(cf. \cite[Theorem 8.6, p.137]{EK}) to prove it.

On one hand, by (H4) and It\^o's formula, we have
\begin{align*}
  \E_{R_n}\Phi(\La_t )&=\Phi(\theta)+\E_{R_n}\int_0^tQ_{\mu_s }\Phi(\La_s )\d s\\
  &\leq \Phi(\theta)+ \E_{R_n}\int_0^t\big(\lambda\Phi(\La_s )+\kappa_0\big)\d s.
\end{align*}
Then Gronwall's inequality leads to that
\begin{equation}\label{n-1}
\E_{R_n} \Phi(\La_t )\leq \big(\Phi(\theta)+\kappa_0T\big)\e^{\lambda t}, \quad t\in [0,T].
\end{equation}
Then, for any $\veps>0$, take $N_\veps$ large enough so that
\[\frac{\E_{R_n}\Phi(\La_t )}{N_\veps}\leq \frac{\big(\Phi(\theta)+\kappa_0T\big)\e^{\lambda T}}{N_\veps}<\veps.\]
Let
\[K_\veps=\big\{\gamma\in\S;\ \Phi(\gamma)\leq N_\veps\big\},\]
which is a compact set because $\Phi$ is a compact function.
Then,
\begin{equation}\label{n-2}
\sup_n{R_n}\big(\La_t\in K_\veps^c\big)\leq \sup_n\frac{\E_{R_n}\Phi(\La_t )}{N_\veps}<\veps.
\end{equation}

On the other hand, we also need to show that for any $\delta>0$ there exists a nonnegative random variable $\gamma_n(\delta)\geq 0$ such that
\[\E_{R_n}\big[\mathbf{1}_{\La_{t+u}\neq\La_t}|\mathscr F_t\big]\leq \E_{R_n}[\gamma_n(\delta)|\mathscr F_t], \quad 0\leq t\leq T, \ 0\leq u\leq \delta,\]
and $\lim_{\delta\ra 0}\sup_n\E_{R_n}[\gamma_n(\delta)]=0$.
Under (H2), the transition rate $(q(\theta,A;u),q(\theta;u))$ of $(\La_t)$ is bounded, and hence
\[R_n\big(\La_s =\La_t ,\, \forall\,s\in[t,t+u]\big)\geq \E_{R_n}\Big[\exp\Big(\!-\!\int_t^{t+u}\!\!\!\sup_{\theta\in \S} q(\theta;\mu_s )\d s\Big)\Big]\geq \exp(-Mu).\]
Then, for every $0\leq u\leq \delta$,
\begin{align*}
  \E_{R_n}\big[\mathbf 1_{\{\La_{t+u}\neq \La_t\}}\big]
   & \leq 1-R_n(\La_s=\La_t,\,\forall\,s\in[t,t+u]) \\
   & \leq 1-\e^{-Mu}\leq 1-\e^{-M\delta}=:\gamma_n(\delta).
\end{align*} It is clear that $\lim_{\delta\ra 0}\sup_n\E_{R_n}\gamma_n(\delta)=0$. Combining this with \eqref{n-2}, we conclude that $(\ll_\La^n)_{n\geq 1}$ is tight.

As a consequence,
the fact $(\ll_\La^n)_{n\geq 1}$ and $(\ll_\mu^n)_{n\geq 1}$ are both tight leads to that for any $\veps>0$, there exist compact sets $K_1\subset C([0,T]; \pb(U))$ and $K_2\subset \mathcal D([0,T];\S)$ such that
\begin{gather*}
R_n\big(\mathcal D([0,T];\S)\times\!K_1^c\big)=\ll_{\mu}^{n}\big(K_1^c\big)<\veps,\\
R_n\big(K_2^c\times \pb([0,T]\times U)\big)=\ll_{\La}^n\big(K_2^c\big)<\veps,
\end{gather*} where $K_i^c$, $i=1,2$, stands for the complement of $K_i$. So,
\[R_n(\big(K_1\times K_2\big)^c\big)\leq R_n\big(\mathcal D([0,T];\S)\times\!K_1^c\big)+R_n\big(K_2^c\times \pb([0,T]\times U)\big)< 2\veps,\]
which implies the desired tightness of $(R_n)_{n\geq 1}$.

\textbf{Step 2}. We go to show the existence of the optimal $\psi$-relaxed control in this step.
According to the result of \textbf{Step 1}, $(R_n)_{n\geq 1}$ is tight, and up to taking a subsequence, $R_n$ converges weakly to some probability measure $R_0$ on $\mathcal Y$.  According to Skorokhod's representation theorem (cf. \cite[Chapter 3]{EK}, Theorem 1.8, p. 102), there exists a probability space $(\Omega',\F',\p')$ on which are defined $\mathcal Y$-valued random variables $Y_n=(\La_t^{(n)},\mu_t^{(n)})_{t\in[0,T]}$, $n=1,2,\ldots$, and $Y_0=(\La_t^{(0)},\mu_t^{(0)})_{t\in[0,T]}$ with distribution $R_n$, $n=1,2,\ldots$, and $R_0$ respectively such that
\begin{equation}\label{n-3}
  \lim_{n\ra \infty} Y_n=Y_0,\quad \text{$\p'$-a.s.}.
\end{equation}

Denote $\F_t'$ the natural $\sigma$-field generated by $(\La_s^{(n)},\mu_s^{(n)})$, $n=0,1,2,\ldots$, up to time $t$. We shall prove that $\alpha^\ast=(\Omega',\F',\F_t',\p',\La_t^{(0)},\mu_t^{(0)},0,\theta)$ is an optimal $\psi$-relaxed control with respect to the value function $V(0,\theta)$. To this end, we need to check
that $\alpha^\ast$ satisfies the conditions of Definition \ref{def-1}. Obviously, conditions (1) and (2) of Definition \ref{def-1} hold.

To check condition (4), the transition semigroup of $(\La_t^{(n)})$, $P_{s,t}^{\mu^{(n)}}\mathbf{1}_{A}(\theta'):=\p'(\La_t^{(n)}\in A|\La_s^{(n)}=\theta')$, $\theta'\in \S$, $A\in \B(\S)$, is determined by the equation \eqref{m-9} with $Q^{\mu}(t)$ being replaced by $Q^{\mu^{(n)}}(t)$ defined as follows:
\begin{equation}\label{n-5}
\begin{split}
Q^{\mu^{(n)}}(t) h(\theta')&=\E\Big[\int_U\!\int_{\S}\!\!q(\theta',\d \gamma;u)h(\gamma)\mu_t^{(n)}(\d u)\big|\La_t^{(n)}=\theta'\Big]\\
&\qquad-\E\Big[\int_U\!\! q(\theta';u)\mu_t^{(n)}(\d u) h(\theta')\big|\La_t^{(n)}=\theta'\Big].
\end{split}
\end{equation}
Similarly, we can define the operators $P_{s,t}^{\mu^{(0)}}$ and $Q^{\mu^{(0)}}(t)$.

For $0\leq t_1<t_2<\ldots<t_k\leq T$, define the projection map $\pi_{t_1\ldots t_k}:\mathcal D([0,T];\S)\ra \S^k$ by
\[\pi_{t_1\ldots t_k}(\La_\cdot)=(\La_{t_1},\ldots,\La_{t_k}).\]
Let $\mathcal T_0$ consist of those $t\in [0,T]$ for which the projection $\pi_t:\mathcal D([0,T];\S)\ra \S$ is continuous except at points form a set of $R_0$-measure 0. For $t\in [0,T]$, $t\in \mathcal T_0$ if and only if $R_0(J_t)=0$, where
\[J_t=\{\La\in \mathcal D([0,T];\S); \La_t\neq \La_{t-}\}.\]
Also, $0,\,T\in \mathcal{T}_0$ by convention.
As a probability measure on $\mathcal{D}([0,T];\S)$, it is known that the complement of $\mathcal T_0$ in $[0,T]$  is at most countable (cf. \cite[p. 124]{Bill}).  Analogously, define the projection map $\tilde \pi_{t_1\ldots t_k}:\mathcal U\ra \pb(U)^k$ by
\[\tilde \pi_{t_1\ldots t_k}(\mu_\cdot)=(\mu_{t_1},\ldots,\mu_{t_k}),\]
which is clearly continuous.

Since $(\La_t^{(n)},\mu_t^{(n)})_{t\in [0,T]}$ converges almost surely to $(\La_t^{(0)},\mu_t^{(0)})_{t\in [0,T]}$ in the product space $\mathcal{D}([0,T];\S)\times \mathcal{U}$ as $n\ra \infty$ and $\pi_t\times \tilde \pi_t$ is continuous for $t\in \mathcal{T}_0$,
we obtain that $(\La_t^{(n)},\mu_t^{(n)})$ converges almost surely to $(\La_t^{(0)},\mu_t^{(0)})$ for $t\in \mathcal{T}_0$. Since $T\in \mathcal{T}_0$, this implies, in particular, that
\begin{equation}\label{n-5.5}
\La_T^{(n)} \ \text{converges almost surely to $\La^{(0)}_T$ as $n\to \infty$.}
\end{equation}
Letting $n\to \infty$ in \eqref{n-5} for $t\in \mathcal{T}_0$, we obtain
\begin{equation*}
   \lim_{n\ra \infty} Q^{\mu^{(n)}}(t) h(\theta')=Q^{\mu^{(0)}}(t)h(\theta'),\quad h\in \B_b(\S),\ \theta'\in \S.
\end{equation*}
For $t\in \mathcal{T}_0$, it holds
\begin{equation}\label{n-6}
\lim_{n\ra \infty}\p'(\La_{t}^{(n)}\in A|\La_0^{(n)}=\theta)=\p'(\La_t^{(0)}\in A|\La_0^{(0)}=\theta),\quad A\in \B(\S),\theta\in\S.
\end{equation}
Moreover, according to \cite[Theorem 7.8, p.131]{EK}, for every $t\in [0,T]$, there exists a sequence $\{s_n\}_{n\geq 1}$ decreasing to $t$ and $\La_{s_n}^{(n)}$ converges weakly to $\La_t^{(0)}$.
For every $t\in [0,T]$, letting $n\to \infty$ in the following equation
\begin{equation}\label{n-7}
\begin{split}
P_{0,s_n}^{\mu^{(n)}} \!h(\theta')&\!=\!h(\theta')\!+\!\int_0^{s_n}   Q^{\mu^{(n)}}\!(t_1)h(\theta')\d t_1\!+\!\int_0^{s_n}\!\int_{0}^{t_2}\!  Q^{\mu^{(n)}}\!(t_2)  Q^{\mu^{(n)}}\!(t_1)h(\theta')\d t_1\d t_2\\
&\ +\!\sum_{k=3}^\infty\!\int_0^{s_n}\!\!\!\int_0^{t_k}\!\!\!\cdots\!\!\int_0^{t_2}\!  Q^{\mu^{(n)}}\!(t_k) \!  Q^{\mu^{(n)}}\!(t_{k-1})\cdots \! Q^{\mu^{(n)}}\!(t_1)h(\theta')\d t_1\ldots\d t_{k-1}\d t_k,
\end{split}
\end{equation}
we obtain that
\begin{equation}\label{n-8}
\begin{split}
P_{0,t}^{\mu^{(0)}} \!h(\theta')&\!=h(\theta')+\int_0^t   Q^{\mu^{(0)}}(t_1)h(\theta')\d t_1+\int_0^{t}\!\int_{0}^{t_2}\!  Q^{\mu^{(0)}}(t_2)  Q^{\mu^{(0)}}(t_1)h(\theta')\d t_1\d t_2\\
&\ +\!\sum_{k=3}^\infty\!\int_0^{t}\!\!\int_0^{t_k}\!\!\cdots\!\!\int_0^{t_2}  Q^{\mu^{(0)}}(t_k)  Q^{\mu^{(0)}}(t_{k-1})\cdots  Q^{\mu^{(0)}}(t_1)h(\theta')\d t_1\ldots\d t_{k-1}\d t_k.
\end{split}
\end{equation}
Because the right-hand side of \eqref{n-8} is continuous in $t$, we have from \eqref{n-8} that $t\mapsto P_{0,t}^{\mu^{(0)}}h(\theta')$ is continuous.
Whence, \eqref{m-0}, and equivalently \eqref{MD-1}, is satisfied by taking derivative w.r.t.\!\! $t$ in both sides of \eqref{n-8} and taking $h(\theta')=\mathbf{1}_A(\theta')$ for $A\in \B(\S)$.
This means that $(\La_t^{(0)})$ is a continuous-time Markov chain  associated with $(\mu_t^{(0)})$. As a consequence, there is  no $t\in (0,T]$ such that $R_0(J_t)>0$, and hence $\mathcal{T}_0=[0,T]$.

Now we go to check condition (3).
Since $(\mu_t^{(n)})_{t\in[0,T]}$ converges almost surely to $(\mu_t^{(0)})_{t\in[0,T]}$ in $\mathcal{D}([0,T];\pb(U))$, we have for each $t\in[0,T]$, $\mu_t^{(n)}$ converges almost surely to $\mu_t^{(0)}$ since $\mathcal{T}_0$ associated with $(\mu_t^{(0)})_{t\in[0,T]}$ equals to $[0,T]$. 
We adopt the notation in the study of backward martingale to define the filtration with negative indices. Let $\F_{-n}^\La=\overline{\sigma(\La_t^{(m)},m\geq n)}$, the completion of the $\sigma$-field generated by $\La_t^{(m)}, m\geq n$. Then
\[\F_{-1}^\La\supset\F_{-2}^\La\supset\cdots\supset \F_{-n}^\La\supset\F_{-n-1}^\La\supset\cdots.\] Put $\F_{-\infty}^\La=\bigcap_{n\geq 1}\F_{-n}^\La$.  $\F_{-\infty}^\La$ is easily checked to be a $\sigma$-field which concerns only the limit behavior of the sequence $\La_t^{(n)}$, $n\geq 1$. Moreover, since there is no point in $[0,T]$ such that $(\La_t^{(0)})$ must jump at that point with positive probability.  Therefore, $\lim_{n\ra \infty}\La^{(n)}_t=\La_t^{(0)}$ a.s. for every $t\in [0,T]$, and further \[\F_{-\infty}^\La=\overline{\sigma(\La_t^{(0)})}.\]
Define $\F_{-n}^\mu=\overline{\sigma(\mu_t^{(m)}, m\geq n)}$. Due to Definition \ref{def-1} (3), $\mu_t^{(n)}$ is in $\F_{-n}^{\La}$ for each $n\geq 1$, and hence $\F_{-n}^\mu\subset \F_{-n}^\La$. Therefore, it follows from the fact $\lim_{n\ra \infty}W_1(\mu_t^{(n)},\mu_t^{(0)})=0$ a.s. that
\[\overline{\sigma(\mu_t^{(0)})}\subset \bigcap_{n\geq 1}\F_{-n}^\mu\subset \F_{-\infty}^\La=\overline{\sigma(\La_t^{(0)})},\]
which means that $\mu_t^{(0)}$ is adapted to $\sigma(\La_t^{(0)})$.

\textbf{Step 3}. Invoking \eqref{n-5.5}, \eqref{n-3}, \eqref{vn}, and \eqref{m-1},  we obtain by the lower semi-continuity of $f$ and $g$ that
\begin{equation}\label{n-4}
\begin{split}
  V(0,\theta)&=\lim_{n\ra \infty} \E_{\p'}\Big[\int_0^T\!\!f(t,\La_t^{(n)},\mu_t^{(n)})\d t+g(\La_T^{(n)})\Big]\\
  &=\lim_{n\ra \infty}\E_{\p'}\Big[\int_0^T\!\int_U\!\!f(t,\La_t^{(n)},u)\mu_t^{(n)}(\d u)\d t+g(\La_T^{(n)})\Big]\\
  &\geq \E_{\p'}\Big[\int_0^T\!\int_U\!\! f(t,\La_t^{(0)},u)\mu_t^{(0)}(\d u)\d t+g(\La_T^{(0)})\Big]\\
  &\geq V(0,\theta).
\end{split}
\end{equation} Hence, $\alpha^\ast$ is an optimal $\psi$-relaxed control.
The proof of this theorem is completed. \fin


After the existence of optimal $\psi$-relaxed control has been established, it is easy to use the time shift technique to prove the continuous property of the value function $V(s,\theta)$ under suitable condition of the cost functions; see the argument of Proposition \ref{t2} in a more complicated situation.
Moreover, based on the Dynkin formula, we can get a lower bound of the value function as follows.
Suppose there exists a measurable function $\varphi:[0,T]\times \S\ra \R$ satisfying $t\mapsto \varphi(t,\theta)$ is differentiable and
\begin{gather*}
  \varphi'(t,\theta)+f(t,\theta,u)+\sum_{\ell\in\S}q(\theta,\{\ell\};u)\varphi(t,\ell)
  -q(\theta;u)\varphi(t,\theta)\geq 0,\\
  \varphi(T,\theta)=g(\theta),
\end{gather*} for every $t\in [0,T]$, $\theta\in\S$, $u\in U$. Then
\[V(s,\theta)\geq \varphi(s,\theta),\quad s\in [0,T],\ \theta\in\S.\]
See, for example, \cite[Section 3]{Guo15} for more details.


\section{Optimal Markov control for CTMDPs in a random environment}

In this section, we consider the random impact of the environment to CTMDPs. In such situation, the cost function depends not only on the paths of continuous-time Markov chains, but also on a stochastic process used to characterize, for instance, the price of raw materials. Precisely, such a dynamical system consists of two components: a diffusion process $(X_t)$ and a continuous-time Markov chain $(\La_t)$, which is also called a regime-switching diffusion process; see, \cite{MY} and \cite{YZ} and references therein.
The process $(X_t)$ is determined by the following SDE:
\begin{equation}\label{3.1}
\d X_t=b(X_t,\La_t)\d t+\sigma(X_t,\La_t)\d B_t,
\end{equation} where $(B_t)$ is a Brownian motion in $\R^d$;
$(\La_t)$ is a continuous-time Markov process on the state space $\S$ associated with the $q$-pair $(q(\theta;u)$, $ q(\theta, A;u))$ satisfying
\begin{equation}\label{3.2}
\p(\La_{t+\delta}\in A|\La_t=\theta, \mu_t=\mu)-\mathbf{1}_{A}(\theta)=\big(q(\theta,A;\mu)
    -q(\theta;\mu)\mathbf{1}_{A}(\theta)\big)\delta+o(\delta)
\end{equation} provided $\delta>0$. The decision-maker still tries to minimize the cost through controlling the transition rates of the Markov chain $(\La_t)$, but now the cost function may depend on the diffusion process $(X_t)$.   Such kind of control problem is quite different to the usual studied optimal controls for SDEs (see, e.g. \cite{HL90, HS95}) or optimal controls for SDEs with regime-switching (see, e.g. \cite{SSZ,SZ16, ZY03}),
where the control policies are placed directly to the drifts or diffusion coefficients of $(X_t)$. Namely,
the controlled system is also given by
\begin{equation}\label{3.3}
\d \wt X_t=b(\wt X_t,\mu_t)\d t+\sigma(\wt X_t,\mu_t)\d B_t.
\end{equation}
%
Roughly speaking, for $(\wt X_t)$, if we change the value of the control $\mu_t$ at time $t$, then the speed of $\wt X_t$ is immediately modified. Nevertheless, for $(X_t)$ given by \eqref{3.1}, if we change $\mu_t$ at time $t$, we only change the switching rate of the process $(\La_t)$ and the speed of $X_t$ maybe remain the same as before because $\La_t$ may not jump at $t$. This observation tells us that in contrast to the process $(\wt X_t)$, the process $(X_t)$ characterized by \eqref{3.1} and \eqref{3.2} is more closely related to the long time behavior of the control $(\mu_t)$.

Let $\psi$, $w_\mu''(\delta)$ be defined by \eqref{psi} and \eqref{w2} respectively.
\begin{mydefn}\label{def-2}
A $\psi$-relaxed control is a term $\alpha=(\Omega,\F,\F_t,\p,B_t,X_t,\La_t,\mu_t,s,x,\theta)$ such that
\begin{itemize}
  \item[$\mathrm{(1)}$] $(s,x,\theta)\in [0,T]\times \R^d\times\S$;
  \item[$\mathrm{(2)}$] $(\Omega,\F,\p)$ is a probability space with the filtration $\{\F_t\}_{t\in [0,T]}$;
  \item[$\mathrm{(3)}$] $(B_t)$ is a $d$-dimensional standard Brownian motion on $(\Omega,\F,\F_t,\p)$, and $(X_t,\La_t)$ is a stochastic process on $\R^d\times \S$ satisfying \eqref{3.1} and \eqref{3.2} with $X_s=x$, $\La_s=\theta$;
  \item[$\mathrm{(4)}$] $\mu_t\in \pb(U)$ is adapted to the $\sigma$-field generated by
      $\La_t$,   $t\mapsto\!\mu_t$ is in $\mathcal{D}([0,T];\!\pb(U))$ almost surely, and for every $\theta'\in \S$ the curve $t\mapsto\nu_t(\,\cdot,\theta'\,):=\mu_t(\, \cdot\, |\La_t=\theta')$ satisfies
      \[w_{\nu}([t_1,t_2))\leq \psi(t_2-t_1),\quad 0\leq t_1<t_2\leq T;\]
\end{itemize}
\end{mydefn}
The collection of all $\psi$-relaxed controls with initial value $(s,x,\theta)$ is denoted by $\wt \Pi_{s,x,\theta}$.

\begin{myrem}
In Definition \ref{def-2}(4), the control policy $\mu_t$ is assumed to be adapted to the $\sigma$-field generated by $\La_t$ in order to ensure the controlled process $(\La_t)$ remain to be a Markov chain. In realistic application, one may make decision using the information of $X_t$. In that case, we naturally need to assume $\mu_t$ is adapted to the $\sigma$-field generated by $\La_t$ and $X_t$. But, $\La_t$ is no longer a Markov process. 
\end{myrem}

Given two functions $f:[0,T]\times \R^d\times\S\times U\to \R$ and $g:\R^d\times \S\to \R $, the expected cost relative to the control $\alpha\in \wt \Pi_{s,x,\theta}$ is defined by
\begin{equation}\label{3.4}
J(s,x,\theta,\alpha)=\E \Big[\int_s^T f(t,X_t,\La_t,\mu_t)\d t+g(X_T,\La_T)\Big].
\end{equation}
Correspondingly, the value function is defined by
\begin{equation}\label{vfun1}
V(s,x,\theta)=\inf_{ \alpha\in \wt \Pi_{s,x,\theta}} J(s,x,\theta,\alpha)
\end{equation}
for  $s\in[0,T],\, x\in\R^d,\, \theta\in\S$.
A $\psi$-relaxed control $\alpha^\ast\in \wt \Pi_{s,x,\theta}$ is called optimal, if  it holds
\[V(s,x,\theta)=J(s,x,\theta,\alpha^\ast).\]

We assume that the coefficients of \eqref{3.1} satisfy the following conditions.
\begin{itemize}
    \item[(H5)] There exists a constant $C_1>0$ such that
  \[|b(x,\theta)-b(y,\theta)|^2+\|\sigma(x,\theta)-\sigma(y,\theta)\|^2\leq C_1|x-y|^2,\ \ \ x,y\in\R^d,\ \theta\in\S,\]
  where $|x|^2=\sum_{k=1}^d x_k^2$, $\|\sigma\|^2=\mathrm{tr}(\sigma\sigma')$, and $\sigma'$ is the transpose of the matrix $\sigma$.
  \item[(H6)] There exists a constant $C_2>0$ such that
      \[|b(x,\theta)|^2+\|\sigma(x,\theta)\|^2\leq C_2(1+|x|^2),\quad x\in \R^d,\theta\in \S.\]
\end{itemize}
The conditions $\mathrm{(H5)}$ and  $\mathrm{(H6)}$ are classical conditions to ensure the existence and uniqueness of nonexplosive solution of SDE \eqref{1.1}. These conditions can be weakened to include some non-Lipschitz coefficients (cf. e.g. \cite{Sh15c}) or singular coefficients (cf. e.g. \cite{KR05}).

Our second main result of this work is the following theorem.

\begin{mythm}\label{t1}
Assume that (H1)-(H6) hold, and
$f$ and $g$ are lower semi-continuous and bounded from below. Then for every $s\in [0,T)$, $x\in \R^d$, $\theta\in \S$, there exists an optimal $\psi$-relaxed control $\alpha^\ast\in \wt \Pi_{s,x,\theta}$.
\end{mythm}

To simplify the proof, we also transform the relaxed controls into the canonical path space.
Let $\mathcal{U}$ be defined by \eqref{U}, and
\begin{equation}\label{spY}
\mathcal Y= C([0,T];\R^d)\times \mathcal D([0,T];\S)\times \mathcal U,
\end{equation}
endowed with the product topology. Let  $\wt{\mathcal Y}$ be the Borel $\sigma$-field, $\wt {\mathcal{Y}}_t$ the $\sigma$-fields up to time $t$.
Now, the relaxed control $(\Omega,\F,\F_t,\p$, $B_t$, $X_t,\La_t,\mu_t,s, x,\theta)$ can be transformed into a relaxed control in the canonical space $\mathcal Y$ via the map $\Psi:\Omega\ra \mathcal Y$ defined by
\[\Psi(\omega)=(X_t(\omega),\La_t(\omega),\mu_t(\omega))_{t\in [0,T]},\quad X_r:=x,\ \La_r:=\theta,\mu_r:=\mu_s\quad \forall \,r\in [0,s],\]
where $R=\p\circ \Psi^{-1}$ is a probability measure on $\mathcal Y$. In this canonical space, we still use $R$ to represent this relaxed control $(\mathcal Y,\wt{\mathcal Y},\wt{\mathcal Y}_t,R,B_t, X_t,\La_t,\mu_t, s,x,\theta)$.

\noindent\textbf{Proof of Theorem \ref{t1}}
Without loss of generality, we consider the case $V(0,x,\theta)<\infty$. In the canonical space $\mathcal Y$, there exists a sequence of probability measures $R_n$, $n\geq 1$, such that
\begin{equation}\label{j-1}
\lim_{n\ra \infty} J(0,x,\theta,R_n)=V(0,x,\theta)<\infty.
\end{equation}

\textbf{Step 1}. In this step, we aim to prove the tightness of $(R_n)_{n\geq 1}$. Denote by $\ll_X^n$, $\ll_\La^n$ and $\ll_\mu^n$, $n\geq 1$, the distribution of $(X_t)_{t\in [0,T]}$, $(\La_t)_{t\in [0,T]}$ and $(\mu_t)_{t\in [0,T]}$ respectively under $R_n$.

In the same way as the proof of Theorem \ref{tm-1}, we can prove the tightness of $(\ll_\mu^n)_{n\geq 1}$ and $(\ll_\La^n)_{n\geq 1}$. Now, we go to prove the tightness of $(\ll_X^n)$. According to
\cite[Theorem 12.3]{Bill}, it is sufficient to verify the moment condition. By It\^o's formula, for $0\leq t_1<t_2\leq T$,
\begin{equation}\label{j-1.5}
\begin{split}
  &\E_{R_n}|X_{t_2}-X_{t_1}|^4\\
  &\leq 8\E_{R_n}\Big|\int_{t_1}^{t_2}\!b(X_r,\La_r)\d r\Big|^4+8\E_{R_n}\Big|\int_{t_1}^{t_2}\!\sigma(X_r,\La_r)\d B_r\Big|^4\\
  &\leq 8(t_2-t_1)^3\E_{R_n}\int_{t_1}^{t_2}\!|b(X_r,\La_r)|^4\d r+288(t_2-t_1)\E_{R_n}\!\int_{t_1}^{t_2}\!\!\|\sigma(X_r,\La_r)\|^4\d r\\
  &\leq C(t_2-t_1)\!\int_{t_1}^{t_2}\!(1+\E_{R_n}|X_r|^4)\d r.
\end{split}
\end{equation}
The linear growth condition (H6) implies  the existence of a constant $C$ (independent of $n$) such that
$\dis \int_0^T\!\!\!\E_{R_n}\!|X_r|^4\d r\leq C$ (cf. \cite[Theorem 3.20]{MY}). Furthermore, invoking the fact $X_0=x$, we conclude that $(\ll_X^n)_{n\geq 1}$ is tight due to
\cite[Theorem 12.3]{Bill}.

\textbf{Step 2}. Because the marginal distributions of $R_n$, $n\geq 1$ are all tight, we get $R_n$, $n\geq 1$ is tight as well. Up to taking a subsequence, we may assume that $R_n$ weakly converges to some probability measure $R_0$. Since $\mathcal Y$ is a Polish space, we apply Skorokhod's representation theorem (cf. \cite[Chapter 3]{EK}, Theorem 1.8, p.102) to obtain a probability space $(\Omega',\F',\p')$ on which defined a sequence of random variables $(X_t^{(n)},\La_t^{(n)},\mu_t^{(n)})_{t\in [0,T]}$, $n\geq 0$, taking values in $\mathcal Y$ with the distribution $R_n$, $n\geq 0$, respectively, such that $(X_t^{(n)},\La_t^{(n)},\mu_t^{(n)})_{t\in [0,T]}$ converges $\p'$-almost surely to $(X_t^{(0)},\La_t^{(0)},\mu_t^{(0)})_{t\in [0,T]}$ as $n\ra \infty$.

Let $\mathcal{T}_0$ be defined in the same way as the argument of Theorem \ref{tm-1}. For every $t\in \mathcal{T}_0$, we have $(X_t^{(n)}, \La_t^{(n)},\mu_t^{(n)})$ converges almost surely to $(X_t^{(0)},\La_t^{(0)},\mu_t^{(0)})$.
Analogous to the argument of Theorem \ref{tm-1}, $(\La_t^{(0)})$ is a continuous time Markov chain with transition rate operator induced from $(\mu_t^{(0)})$, which also implies that $\mathcal{T}_0=[0,T]$. The fact that $\mu_t^{(0)}$ is adapted to $\sigma(\La_t^{(0)})$ can be proved in the same way as the proof of Theorem \ref{tm-1}.

We need to check that $(X_t^{(0)})$ satisfies SDE \eqref{3.1} under $R_0$ is associated with a $\psi$-relaxed control.
Since $(X_t^{(n)})_{t\in [0,T]}$ are processes in the path space $C([0,T];\R^d)$, every projection map $\pi_t:C([0,T];\R^d)\ra \R^d$, $\pi_t(X_\cdot):=X_t$, is continuous. Then, this yields that
\[X_t^{(n)}\ \text{converges almost surely to $X_t^{(0)}$ for each $t\in [0,T]$ as $n\ra \infty$,}
\] because $(X_t^{(n)}, \La_t^{(n)},\mu_t^{(n)})_{t\in [0,T]}$ converges almost surely to $(X_t^{(0)},\La_t^{(0)},\mu_t^{(0)})_{t\in[0,T]}$.
Furthermore, passing $n$ to $\infty$ in the following integral equation:
\begin{equation}\label{j-2}
X_t^{(n)}=x+\int_0^t b(X_s^{(n)},\La_s^{(n)})\d s+\int_0^t\sigma(X_s^{(n)},\La_s^{(n)})\d B_s.
\end{equation}
we get
\begin{equation}\label{j-3}
X_t^{(0)}=x+\int_0^t b(X_s^{(0)},\La_s^{(0)})\d s+\int_0^t\sigma(X_s^{(0)},\La_s^{(0)})\d B_s,
\end{equation}
which means that $(X_t^{(0)})$ satisfies SDE \eqref{3.1}.

Consequently,  $R_0$ is a $\psi$-relaxed control. By  \eqref{j-1} and the lower semi-continuity of $f$ and $g$,
we have
\begin{align*}
  V(0,x,\theta)&=\lim_{n\ra \infty}\E_{\p'}\Big[\int_0^T f(t,X_t^{(n)},\La_t^{(n)},\mu_t^{(n)})\d t+g(X_T^{(n)},\La_T^{(n)})\Big]\\
  &\geq \E_{\p'}\Big[\int_0^T f(t,X_t^{(0)},\La_t^{(0)},\mu_t^{(0)})\d t+g(X_T^{(0)},\La_T^{(0)})\Big]\\
  &\geq V(0,x,\theta).
\end{align*}
Hence, $R_0$ is an optimal $\psi$-relaxed control.
The proof is complete. \fin

\end{document}